\documentclass[reqno]{amsart}

\usepackage{amsmath}
\usepackage{amsfonts}
\usepackage{amssymb,enumerate}
\usepackage{amsthm,stmaryrd}
\usepackage[all]{xy}
\usepackage[draft=false]{hyperref}
\usepackage{tikz}
\usetikzlibrary{matrix,arrows,decorations.pathmorphing}
\usepackage{comment}
\usepackage[obeyDraft,color=green!40]{todonotes}
\theoremstyle{plain}
\newtheorem{lem}{Lemma}[section]

\newtheorem{thm}[lem]{Theorem}

\theoremstyle{definition}

\newtheorem{fact}[lem]{Fact}

\newtheorem{assumption}[lem]{Assumption}

\numberwithin{equation}{lem}

\renewenvironment{proof}{\vspace{1ex}\noindent{\textbf{Proof:}}\hspace{0.5em}}
{\hfill\qed\vspace{1ex}}

\newcommand{\D}{\mathcal{D}}

\newcommand{\cat}[1]{\mathcal{#1}}

\newcommand{\catf}{\cat{F}}
\newcommand{\cati}{\cat{I}}

\newcommand{\catac}{\cat{A}_C}

\newcommand{\catbc}{\cat{B}_C}

\newcommand{\catic}{\cat{I}_C}

\newcommand{\catpc}{\cat{P}_C}
\newcommand{\catfc}{\cat{F}_C}

\newcommand{\pd}{\operatorname{pd}}

\newcommand{\id}{\operatorname{id}}	
\newcommand{\fd}{\operatorname{fd}}

\newcommand{\pcpd}{\catpc\text{-}\pd}
\newcommand{\fcpd}{\catfc\text{-}\pd}

\newcommand{\icid}{\catic\text{-}\id}

\newcommand{\depth}{\operatorname{depth}}

\newcommand{\ext}{\operatorname{Ext}}	
\newcommand{\rhom}{\mathbf{R}\!\operatorname{Hom}}	
\newcommand{\lotimes}{\otimes^{\mathbf{L}}}

\newcommand{\Hom}{\operatorname{Hom}}	

\newcommand{\spec}{\operatorname{Spec}}

\newcommand{\shift}{\mathsf{\Sigma}}

\newcommand{\ideal}[1]{\mathfrak{#1}}
\newcommand{\m}{\ideal{m}}

\newcommand{\p}{\ideal{p}}

\newcommand{\supp}{\operatorname{supp}}

\newcommand{\bbz}{\mathbb{Z}}

\newcommand{\xra}{\xrightarrow}

\newcommand{\x}{\mathbf{x}}

\renewcommand{\geq}{\geqslant}

\renewcommand{\hom}{\Hom}

\begin{document}

\title{Using semidualizing complexes to detect Gorenstein rings}

\thanks{Sather-Wagstaff was supported in part by a grant from the NSA. Totushek was supported in part by North Dakota EPSCoR and National Science Foundation Grant EPS-0814442}



\author{Sean Sather-Wagstaff}
\address{Department of Mathematics,
NDSU Dept \# 2750,
PO Box 6050,
Fargo, ND 58108-6050
USA}

\email{sean.sather-wagstaff@ndsu.edu}

\urladdr{http://www.ndsu.edu/pubweb/\~{}ssatherw/}
\author{Jonathan Totushek}
\email{jonathan.totushek@ndsu.edu}

\urladdr{http://www.ndsu.edu/pubweb/\~{}totushek/}

\begin{abstract}
A result of Foxby states that if there exists a complex with finite depth, finite flat dimension, and finite injective dimension over a local ring $R$, then $R$ is Gorenstein. In this paper we investigate some homological dimensions involving a semidualizing complex and improve on Foxby's result by answering a question of Takahashi and White. In particular, we prove for a semidualizing complex $C$, if there exists a complex with finite depth, finite $\catf_C$-projective dimension, and finite $\catic$-injective dimension over a local ring $R$, then $R$ is Gorenstein.
\end{abstract}

\maketitle

\section{Introduction}

Throughout this paper let $R$ be a commutative noetherian ring with identity. A result of Foxby states that, if there exists an $R$-complex $X$ that has finite flat dimension and finite injective dimension, then $R_{\p}$ is a Gorenstein ring for all $\p\in\supp_R(X)$; see Section \ref{141222:3} for definitions. In this paper we generalize this theorem using a semidualizing $R$-module.
A finitely generated $R$-module $C$ is \textit{semidualizing} if $R\cong \hom_R(C,C)$ and $\ext^{\geq1}_R(C,C) = 0$. Semidualizing modules are useful, e.g., for proving results about Bass numbers \cite{avramov:rhafgd,sather:bnsc} and compositions of local ring homomorphisms \cite{avramov:rhafgd,sather:cidfc}. 

Takahashi and White \cite{takahashi:hasm} define the $C$-projective dimension for an $R$-module $M$ (denoted $\pcpd_R(M)$) to be the length of the shortest resolution by modules of the form $C\otimes_R P$ where $P$ is a projective $R$-module. They define $C$-injective dimension ($\icid$) dually. 
In their investigation Takahashi and White posed the following question: When $R$ is a local Cohen-Macaulay ring admitting a dualizing module and $C$ is a semidualizing $R$-module, if there exists an $R$-module $M$ such that  $\pcpd_R(M)<\infty$ and $\icid_R(M)<\infty$, must $R$ be Gorenstein? If $M$ has infinite depth, then the answer is false. However, if we additionally assume that $M$ has finite depth, then an affirmative answer to this question would yield a generalization of Foxby's theorem.

Partial answers to Takahashi and White's question is given by Araya and Takahashi \cite{araya:agoatof} and Sather-Wagstaff and Yassemi \cite{sather:mfhdsdm}. We give a complete answer to this question in the following result; see Section \ref{141222:3} for background on complexes and the derived category, and Theorem \ref{1405267:1} for the proof.

\begin{thm}
\label{141215:1}
	Let $C$ be a semidualizing $R$-complex. If there exists an $R$-complex $X\in \D_{\operatorname{b}}(R)$ such that $\fcpd_R(X)<\infty$ and $\icid_R(X)<\infty$, then $R_{\p}$ is Gorenstein for all $\p\in \supp_R(X)$.
\end{thm}

%Moreover, our main result is a more general version where $C$ is a semidualizing $R$-complex and $M$ is a homologically bounded $R$-complex. Accordingly, the proof relies on derived category techniques, which we summarize in the next section.

\section{Background}\label{141222:3}

Let $\D(R)$ denote the derived category of complexes of $R$-modules, indexed homologically (see e.g. \cite{gelfand:moha,hartshorne:rad}). A complex $X\in \D(R)$ is \textit{homologically bounded below}, denoted $X\in \D_{\operatorname{+}}(R)$, if $\operatorname{H}_i(X)=0$ for all $i\ll 0$. It is \textit{homologically bounded above}, denoted $X\in\D_{-}(R)$, if $\operatorname{H}_i(X)=0$ for all $i\gg 0$. It is \textit{homologically degreewise finite}, denoted $X\in\D^{\operatorname{f}}(R)$, if $\operatorname{H}_i(X)$ is finitely generated for all $i$. Set $\D_{\operatorname{b}}(R) = \D_{+}(R) \cap \D_{-}(R)$ and $\D^{\operatorname{f}}_{*}(R)= \D^{\operatorname{f}}(R)\cap \D_{*}(R)$ for each $*\in \{+,-,\text{b}\}$. Complexes in $\D^{\text{f}}_{\text{b}}(R)$ are called \textit{homologically finite}. Isomorphisms in $\D(R)$ are identified by the symbol $\simeq$.

For $R$-complexes $X$ and $Y$, let $\inf(X)$ and $\sup(X)$ denote the infimum and supremum, respectively, of the set $\{i\in \bbz \mid \operatorname{H}_i(X) = 0\}$ with the convention $\sup(\emptyset)=-\infty$ and $\inf(\emptyset)=\infty$. Let $X\lotimes_R Y$ and $\rhom_R(X,Y)$ denote the left-derived tensor product and right-derived homomorphism complexes, respectively.

	If $(R,\m,k)$ is local, the \textit{depth} and \textit{width} of an $R$-complex $X\in \D(R)$ are defined by Foxby~\cite{foxby:bcfm} and Yassemi \cite{yassemi:wcm} as 
	\begin{align*}
		\depth_R(X) &:= -\sup(\rhom_R(k,X))\\
		\operatorname{width}_R(X) &:= \inf(k\lotimes_R X).
	\end{align*}
	One relation between these quantities is given in the following.

	The \textit{small support} of an $R$-complex $X\in \D(R)$ is defined by Foxby~\cite{foxby:bcfm} as follows:
	\[
		\supp_R(X) := \{\p\in \spec(R) \mid \kappa(\p)\lotimes_R X \not\simeq 0\}.
	\]
An important property of the small support is given in the following.

\begin{fact}[{\cite[Proposition 2.7]{foxby:bcfm}}]
\label{140325:1}
	If $X,Y\in \D(R)$, then
	\[
		\supp_R(X\lotimes_R Y) = \supp_R(X)\cap \supp_R(Y).
	\]
\end{fact}

	The \textit{flat dimension} of an $R$-complex $X\in \D_{+}(R)$ is 
	\[
		\fd_R(X) := \inf\left\{n\in \bbz 
		\left|
			\begin{array}{l}
				F\xra{\simeq} X \text{ where } F \text{ is a bounded below complex of}\\
				\text{flat }R \text{-modules such that } F_i=0 \text{ for all } i>n
			\end{array}
		\right.\right\}.
	\]
	The \textit{injective dimension} of an $R$-complex $Y\in \D_{-}(R)$ is  
	\[
		\id_R(X) := \inf\left\{n\in \bbz 
		\left|
			\begin{array}{l}
				Y\xra{\simeq} I \text{ where } I \text{ is a bounded above complex of}\\
				\text{injective }R \text{-modules such that } I_j=0 \text{ for all } j>-n
			\end{array}
		\right.\right\}.
	\]

	A homologically finite $R$-complex $C$ is \textit{semidualizing} if the homothety morphism $\chi^R_C:R\to \rhom_R(C,C)$ is an isomorphism in $\D(R)$. An $R$-complex $D$ is \textit{dualizing} if it is semidualizing and has finite injective dimension. Dualizing complexes were introduced by Grothendieck and Hartshorne \cite{hartshorne:rad}, and semidualizing complexes originate in work of Foxby \cite{foxby:gmarm}, Avramov and Foxby \cite{avramov:rhafgd}, and Christensen \cite{christensen:scatac}. For the non-commutative case, see, e.g., Araya, Takahashi, and Yoshino \cite{takahashi:hiatsb}.

\begin{assumption}
For the rest of this paper, let $C$ be a semidualizing $R$-complex.
\end{assumption}

	The following classes were defined in \cite{avramov:rhafgd,christensen:scatac}.
	The \textit{Auslander Class} with respect to $C$ is the full subcategory $\catac(R)\subseteq \D_{\mathrm{b}}(R)$ such that a complex $X$ is in $\catac(R)$ if and only if $C\lotimes_R X\in \D_{\mathrm{b}}(R)$ and the natural morphism $\gamma^C_X:X\to \rhom_R(C,C\lotimes_R X)$ is an isomorphism in $\D(R)$.
	Dually, the \textit{Bass Class} with respect to $C$ is the full subcategory $\catbc(R)\subseteq \D_{\mathrm{b}}(R)$ such that a complex $Y$ is in $\catbc(R)$ if and only if $\rhom_R(C,Y)\in\D_{\mathrm{b}}(R)$ and the natural morphism $\xi^C_Y:C\lotimes_R \rhom_R(C,Y) \to Y$ is an isomorphism in $\D(R)$.	

The $\catf_C$-projective dimension and $\cati_C$-injective dimension of an $R$-complex $X\in \D_{\text{b}}(R)$ are defined in \cite{totushek:hdwrtasc} as follows:
	\begin{align*}
		\fcpd_R(X) &:= \sup(C) + \fd_R(\rhom_R(C,X))\\
		\icid_R(X) &:= \sup(C) + \id_R(C\lotimes_R X).
	\end{align*}

The following fact shows that the above definitions are consistent with the ones given by Takahashi and White \cite{takahashi:hasm} when $C$ is a semidualizing module.

\begin{fact}[{\cite[Theorem 3.9]{totushek:hdwrtasc}}]
\label{141201:2}
	Let $X\in \D_{\operatorname{b}}(R)$.
	\begin{enumerate}[\rm (a)]
	\item We have $\fcpd_R(X)<\infty$ if and only if there exists an $R$-complex $F\in \D_{\operatorname{b}}(R)$ such that $\fd_R(F)<\infty$ and $X\simeq C\lotimes_R F$. When these conditions are satisfied, one has $F\simeq \rhom_R(C,X)$ and $X\in \catbc(R)$.\label{141201:2a}
	\item We have $\icid_R(X)<\infty$ if and only if there exists an $R$-complex $J\in \D_{\operatorname{b}}(R)$ such that $\id_R(J)<\infty$ and $X\simeq \rhom_R(C,J)$. When these conditions are satisfied, one has $J\simeq C\lotimes_R X$ and $X\in \catac(R)$.\label{141201:2b}
	\end{enumerate}
\end{fact}

\section{Results}\label{150409:1}

The next result fully answers the question of Takahashi and White discussed in the introduction. %For this we require the following two lemmas.

\begin{thm}
\label{140321:3}
	Let $(R,\m,k)$ be a local ring. If there is an $R$-complex $X\in \mathcal{D}_{\operatorname{b}}(R)$ with finite depth, $\fcpd_R(X)<\infty$ and $\icid_R(X)<\infty$, then $R$ is Gorenstein.
\end{thm}

\begin{proof}
	\underline{Case 1}: \textit{$\depth_R(X)<\infty$ and $R$ has a dualizing complex $D$.}

	We first observe that by \cite[Theorem 4.6]{foxby:daafuc} we have that the following:
	\begin{align}
		\depth_R(\rhom_R(C,X)) &= \operatorname{width}_R(C) + \depth_R(X) <\infty.\label{141222:1}
	\end{align}
	Note that $\depth_R(X)$ is finite by assumption, and $\operatorname{width}_R(C)$ is finite by Nakayama's Lemma, as $C$ is homologically finite: see \cite[Lemma 2.1]{foxby:bcfm}.	

	Set $C^{\dagger} := \rhom_R(C,D)$. The assumption $\icid_R(X)<\infty$ with \cite[Theorem~1.2]{totushek:hdwrtasc} implies $\mathcal{F}_{C^{\dagger}}\text{-}\pd_R(X)<\infty$. Hence by Fact \ref{141201:2}\eqref{141201:2a} there exist $R$-complexes $F,G$ of finite flat dimension such that $C\lotimes_R F \simeq X \simeq C^{\dagger}\lotimes_R G$. Since $G$ has finite flat dimension, \cite[Proposotion 4.4]{christensen:scatac} implies $G\in \mathcal{A}_{C^{\dagger}}(R)$, which explains the first isomorphism in the following display:
	\begin{align*}
		G
			\simeq \rhom_R(C^{\dagger},C^{\dagger}\lotimes_R G)
			\simeq \rhom_R(C^{\dagger},C\lotimes_R F)
			\simeq \rhom_R(C^{\dagger},C)\lotimes_R F.
	\end{align*}
	The last isomorphism is by tensor evaluation \cite[Lemma 4.4(F)]{avramov:hdouc}.

	Fact \ref{141201:2}\eqref{141201:2a} implies $F\simeq \rhom_R(C,X)$. By \eqref{141222:1} we have $\depth_R(F)<\infty$.
	It follows from \cite[Proposition 2.8]{foxby:bcfm} that $k\lotimes_R F \not\simeq 0$. Set $U:=\rhom_R(C^{\dagger},C)$. Since $C$ and $C^{\dagger}$ are in $\mathcal{D}_{\text{b}}^{\text{f}}(R)$, we have $U\in \D_{-}^{\text{f}}(R)$.

	\underline{Claim A}: $U\in \D^{\text{f}}_{\text{b}}(R)$.
 
	To prove this claim it suffices to show that $U\in \D_{+}(R)$. Assume by way of contradiction that $\inf(U) = -\infty$. Then by \cite[4.5]{foxby:daafuc} we know that $\inf(k\lotimes_R U) = -\infty$. By tensor cancellation and the K\"{u}nneth formula we have isomorphisms
	\begin{align*}
		\operatorname{H}_{n}\left(k\lotimes_R (F\lotimes_R U) \right)
			&\cong \text{H}_{n}\left( (k\lotimes_R F)\lotimes_k (k\lotimes_R U) \right)\\
			&\cong \bigoplus_{p+q = n}\text{H}_{p}(k\lotimes_R F)\otimes_k \text{H}_q(k\lotimes_R U).
	\end{align*}
	Since $k\lotimes_R F\not\simeq 0$ and $\inf(k\lotimes_R U) = -\infty$ it follows that $\inf(k\lotimes_R (F\lotimes_R U))=-\infty$. On the other hand, since $F\lotimes_R U\simeq G\in \D_{\operatorname{b}}(R)$ we have 
$k\lotimes_R (F\lotimes_R U)\simeq k\lotimes_R G\in \D_{+}(R)$, so $\inf(k\lotimes_R (F\lotimes_R U))>-\infty$, a contradiction. This establishes Claim A.

	\underline{Claim B}: The complex $U$ has finite projective dimension.

	To show this claim assume by way of contradiction that $\pd_R(U)=\infty$. Then because $U\in\D^{\text{f}}_{\text{b}}(R)$ we have $\sup(k\lotimes_R U) = \infty$ by \cite[Proposition 5.5]{avramov:hdouc}. As in the proof of Claim A, we conclude that $\sup(k\lotimes_R (F\lotimes_R U)) = \infty$. On the other hand, we have $k\lotimes_R(F\lotimes_R U)\simeq k\lotimes_R G$. Since $G$ has finite flat dimension, this implies that $\sup(k\lotimes_R (F\lotimes_R U))<\infty$, a contradiction. This concludes the proof of Claim~B.

	Now \cite[Theorem 1.4]{frankild:rbsc} implies that $\shift^n C\simeq C^{\dagger}= \rhom_R(C,D)$ for some $n\in \bbz$. Hence by \cite[Corollary 3.4]{frankild:sdcms} we deduce that $R$ is Gorenstein. This concludes the proof of Case 1.
	
	\underline{Case 2}: $\supp_R(X) = \{\m\}$.

	For the proof of Case 2, first observe that $R$ is Gorenstein if and only if $\widehat{R}$ is Gorenstein. Since $\widehat{R}$ has a dualizing complex, by Case 1 it suffices to show that
	\begin{enumerate}[(1)]
	\item $\widehat{R}\lotimes_R X\in \D_{\text{b}}(\widehat{R})$,
	\item $\widehat{R}\lotimes_R C$ is a semidualizing $\widehat{R}$-complex,
	\item $\mathcal{F}_{\widehat{R}\lotimes_R C}\text{-}\pd_{\widehat{R}}(\widehat{R}\lotimes_R X)<\infty$,
	\item $\mathcal{I}_{\widehat{R}\lotimes_R C}\text{-}\id_{\widehat{R}}(\widehat{R}\lotimes_R X)<\infty$, and
	\item $\depth_{\widehat{R}}(\widehat{R}\lotimes_R X)<\infty$.
	\end{enumerate}

	Observe that (1) follows from the fact that $\widehat{R}$ is flat over $R$. Items (2) and (3) follow from \cite[Lemma 2.6]{christensen:scatac} and \cite[Proposition 3.11]{totushek:hdwrtasc}, respectively.
	
	To prove (4) note that the first equality in the next sequence is by definition:
	\begin{align*}
		\mathcal{I}_{\widehat{R}\lotimes_R C}\text{-}\id_{\widehat{R}}(\widehat{R}\lotimes_R X)
			&= \id_{\widehat{R}}\left( (\widehat{R}\lotimes_R C)\lotimes_{\widehat{R}} ( \widehat{R}\lotimes_R X ) \right) + \sup(\widehat{R}\lotimes_R C) \\
			&= \id_{\widehat{R}}\left( \widehat{R}\lotimes_R (C\lotimes_R X) \right) + \sup(\widehat{R}\lotimes_R C).
	\end{align*}
	The second equality is by tensor cancellation. From the condition $\icid_R(X)<\infty$, we have $\id_R(C\lotimes_R X)<\infty$ by definition. Note that $\m\in\spec(R)=\supp_R(C)$ by \cite[Proposition 6.6]{sather:saaffc}. Therefore Fact \ref{140325:1} implies 
	\[
		\supp_R(C\lotimes_R X) = \supp_R(C)\cap \supp_R(X) = \{\m\}.
	\]
	Hence by \cite[Lemma 3.4]{sather:cidfc} the complex $\widehat{R}\lotimes_R(C\lotimes_R X)$ has finite injective dimension over $\widehat{R}$, so (4) holds.

	For the proof of (5) consider the following sequence:
	\begin{align*}
	\depth_{\widehat{R}}(\widehat{R}\lotimes_R X)
		&= -\sup\left( \rhom_{\widehat{R}}(k,\widehat{R}\lotimes_R X )\right)\\
		&= -\sup\left( \rhom_{\widehat{R}}(\widehat{R}\lotimes_R k,\widehat{R}\lotimes_R X) \right)\\
		&= -\sup\left( \widehat{R}\lotimes_R \rhom_{R}(k,X) \right)\\
		&= -\sup\left( \rhom_R(k,X) \right)\\
		&= \depth_R(X) \\
		&< \infty.
	\end{align*}
	The second equality is because $k\cong \widehat{R}\lotimes_R k$, and the fourth equality is because $\widehat{R}$ is faithfully flat over $R$. This establishes (5) and concludes the proof of Case 2.

	\underline{Case 3}: general case.

	Let $\x$ be a generating sequence for $\m$, and let $K= K^R(\x)$ be the Koszul complex. Then $\supp_R(K) = \{\m\}$. Since $\depth_R(X)<\infty$, we have that $\m\in \supp_R(X)$ by \cite[Proposition 2.8]{foxby:bcfm}. Hence, we conclude from Fact \ref{140325:1} that
	\[
		\supp_R(K\lotimes_R X) = \supp_R(K)\cap \supp_R(X) = \{\m\}.
	\]

	By Case 2 it suffices to show that
	\begin{enumerate}[(a)]
	\item $\depth_R(K\lotimes_R X)<\infty$,
	\item $K\lotimes_R X\in \D_{\text{b}}(R)$,
	\item $\fcpd_R(K\lotimes_R X)<\infty$, and
	\item $\icid_R(K\lotimes_R X)<\infty$.
	\end{enumerate}

	Item (a) follows from \cite[Proposition 2.8]{foxby:bcfm}. For (b), use the conditions $\pd_R(K)<\infty$ and $X\in \D_{\text{b}}(R)$. Items (c) and (d) follow from \cite[Proposition 4.5 and 4.7]{totushek:hdwrtasc}. This concludes the proof of Case 3.
\begin{comment}
	For (3), since $\fcpd_R(X)<\infty$, by definition we have $\fd_R(\rhom_R(C,X))<\infty$. Note that $\pd_R(K)<\infty$, therefore  
	\[
		\fd_R(\rhom_R(C,X\lotimes_R K))= \fd_R(\rhom_R(C,X)\lotimes_R K) <\infty. 
	\]
	It now follows that $\fcpd_R(K\lotimes_R X)<\infty$.

	To show (4), by the definition of $\icid$ it suffices to show that $C\lotimes_R (K\lotimes_R X)$ has finite injective dimension. Observe that 
	\[
		C\lotimes_R (K\lotimes_R X) \simeq K\lotimes_R (C\lotimes_R X).
	\]
	Since $\pd_R(K)<\infty$ and $\id_R(C\lotimes_R X)<\infty$ we have that $\id_R(C\lotimes_R (K\lotimes_R X)) = \id_R(K\lotimes_R(C\lotimes_R X))<\infty$. Now case 2 implies that $R$ is Gorenstein.
	\end{comment}
\end{proof}

The following result is Theorem \ref{141215:1} from the introduction.

\begin{thm}\label{1405267:1}
	If there is an $R$-complex $X\in\D_{\operatorname{b}}(R)$ such that $\fcpd_R(X)<\infty$ and $\icid_R(X)<\infty$, then $R_{\p}$ is Gorenstein for all $\p\in\supp_R(X)$.
\end{thm}
\begin{proof}
	By Theorem \ref{140321:3} it suffices to show the following:
	\begin{enumerate}[(i)]
	\item $X_{\p}\in \D_{\text{b}}(R_{\p})$,
	\item $C_{\p}$ is a semidualizing $R_{\p}$-complex,
	\item $\mathcal{F}_{C_{\p}}\text{-}\pd_{R_{\p}}(X_{\p})<\infty$,
	\item $\mathcal{I}_{C_{\p}}\text{-}\id_{R_{\p}}(X_{\p})<\infty$, and
	\item $\depth_{R_\p}(X_{\p})<\infty$.
	\end{enumerate}

	Item (i) follows from the fact that $R_{\p}$ is a flat over $R$, and item (ii) follows from \cite[Lemma 2.5]{christensen:scatac}. Items (iii) and (iv) are by \cite[Corollary 3.12]{totushek:hdwrtasc}.

\begin{comment}
	For (3) it suffices to show that $\fd_{R_{\p}}\left( \rhom_{R_{\p}}(C_{\p},X_{\p} \right)<\infty$. Because $\rhom_{R_{\p}}(C_{\p},X_{\p})\simeq \left( \rhom_R(C,X) \right)_{\p}$ we have
	\[
		\fd_{R_{\p}}(\rhom_{R_{\p}}(C_{\p},X_{\p})) = \fd_{R_{\p}}\left( (\rhom_R(C,X))_{\p} \right).
	\]
	Since $\rhom_R(C,P)$ has finite flat dimension over $R$, $(\rhom_R(C,X))_{\p}$ has finite flat dimension over $R_{\p}$ by flat base change.

	To show (4) it suffices to show that $\id_{R_{\p}}\left( C_{\p}\lotimes_{R_{\p}}X_{\p} \right)<\infty$. Because $C_{\p}\lotimes_{R_{\p}} X_{\p}\simeq (C\lotimes_R X)_{\p}$ we have
	\[
		\id_{R_{\p}}(C_{\p}\lotimes_{R_{\p}}X_{\p}) = \id_{R_{\p}}\left( (C\lotimes_R X)_{\p} \right).
	\]
	Since $C\lotimes_R X$ has finite injective dimension over $R$, we have that $(C\lotimes_R X)_{\p}$ has finite injective dimension over $R_{\p}$ by localization.
\end{comment}

	(5) As $\p\in\supp_R(X)$, we have $\p R_{\p}\in \supp_{R_{\p}}(X_{\p})$ by \cite[Proposition 3.6]{sather:saaffc}. Since $\p R_{\p}$ is the maximal ideal of the local ring $R_{\p}$, we deduce from \cite[Proposition 2.8]{foxby:bcfm} that $\depth_{R_{\p}}(X_{\p}) < \infty$.
\end{proof}

\subsection*{Acknowledgments}
We are grateful to the referee for thoughtful comments and to Ryo Takahashi for bringing \cite{araya:agoatof} to our attention.

%\bibliography{+new}
\bibliographystyle{amsplain}

\providecommand{\bysame}{\leavevmode\hbox to3em{\hrulefill}\thinspace}
\providecommand{\MR}{\relax\ifhmode\unskip\space\fi MR }
\providecommand{\MRhref}[2]{%
  \href{http://www.ams.org/mathscinet-getitem?mr=#1}{#2}
}
\providecommand{\href}[2]{#2}

\end{document}